\documentclass[11pt]{amsart}
\usepackage{amssymb,amscd}
\begin{document}
\newtheorem{thm}{Theorem}[section]
\newtheorem{lem}[thm]{Lemma}
\newtheorem{prop}[thm]{Proposition}
\newtheorem{cor}[thm]{Corollary}
\theoremstyle{definition}
\newtheorem{ex}[thm]{Example}
\newtheorem{rem}[thm]{Remark}
\newtheorem{prob}[thm]{Problem}
\newtheorem{thmA}{Theorem}
\renewcommand{\thethmA}{}
\newtheorem{defi}[thm]{Definition}
\renewcommand{\thedefi}{}
\input amssym.def
\long\def\alert#1{\smallskip{\hskip\parindent\vrule%
\vbox{\advance\hsize-2\parindent\hrule\smallskip\parindent.4\parindent%
\narrower\noindent#1\smallskip\hrule}\vrule\hfill}\smallskip}
\def\ff{\frak}
\def\Spec{\mbox{\rm Spec}}
\def\type{\mbox{ type}}
\def\Hom{\mbox{ Hom}}
\def\rank{\mbox{ rank}}
\def\Ext{\mbox{ Ext}}
\def\Ker{\mbox{ Ker}}
\def\Max{\mbox{\rm Max}}
\def\End{\mbox{\rm End}}
\def\l{\langle\:}
\def\r{\:\rangle}
\def\Rad{\mbox{\rm Rad}}
\def\Zar{\mbox{\rm Zar}}
\def\Supp{\mbox{\rm Supp}}
\def\Rep{\mbox{\rm Rep}}
\def\cal{\mathcal}
\title[On the Chang's group of BL-algebras]{On the Chang's group of BL-algebras}
\thanks{2000 {\it Mathematics Subject Classification.}
Primary 06D99, 08A30}
\thanks{\today}
\author{Clestin Lele and Jean B Nganou}
\address{Department of Mathematics and Computer Science, University of Dschang, Cameroon
} \email{celestinlele@yahoo.com}
\address{Department of Mathematics, University of Oregon, Eugene,
OR 97403} \email{nganou@uoregon.edu}
\begin{abstract} For an arbitrary BL-algebra $L$, we construct an associated lattice Abelian group $G_L$ that coincides with Chang's $\ell$-group when the BL-algebra is an MV-algebra. We prove that the Chang's group of the MV-center of any BL-algebra $L$ is a direct summand in $G_L$. We also compute examples of this group.
 \vspace{0.20in}\\
{\noindent} Key words: BL-algebra, MV-algebra, ideal, MV-center, prime ideal, BL-chain, good sequence, $\ell$-group.
\end{abstract}
\maketitle
\section{Introduction}
On one hand, an MV-algebra is an Abelian monoid $(M,\oplus , 0)$ with an involution $^\ast:M\to M$ (i. e.; $(x^\ast)^\ast=x$ for all $x\in M$) satisfying the following axioms for all $x, y\in M$: $0^\ast\oplus x=0^\ast$;
$(x^\ast\oplus y)^\ast\oplus y=(y^\ast\oplus x)^\ast\oplus x$.
On the other hand, a lattice ordered Abelian group ($\ell$-group) is an Abelian group $(G,+,-,0)$ equipped with a lattice order $\leq$ that is translation invariant. In other words, for all $a,b,c\in G$, if $a\leq b$, then $a+c\leq b+c$.\\
It is known that the category of MV-algebras is equivalent to that of unital lattice ordered groups ($\ell$-groups). This equivalence, which depends in large part on the natural algebraic addition of MV-algebras \cite{CM} has been an essential tool in the study of MV-algebras. One would solve problems of MV-algebras by considering and solving the corresponding problems within lattice ordered groups where more tools are available, and settings are more familiar.  \\
On the other hand, MV-algebras are BL-algebras satisfying the double negation. But, until recently, the essential ingredient (algebraic addition) was still missing within the BL-algebras framework as had observed several authors \cite{a2}, \cite{T3}. Very recently we introduced an algebraic addition in BL-algebras, as natural generalization of the addition in MV-algebras \cite{ln}. \\
Given that a proper algebraic addition is now available in BL-algebras, it is natural to consider the constructions in MV-algebras that rely on the addition. The Chang's $\ell$-group is one of the most important such constructions. We define good sequences in BL-algebras as a natural generalization of good sequences in MV-algebras and show that their set is a commutative monoid under the addition of good sequences. The key step is to prove the associativity of the multiplication of good sequences. To achieve this goal, we use existing properties of MV-chains, the subdirect product representation theorem for BL-algebras, and the decomposition of BL-chains into Wasjberg hoops with bottom hoop being a wasjberg algebra \cite{am, Bu}. In addition, this monoid has a natural structure of a lattice where the suprema and infima are taken component wise. Unlike for MV-algebras, this monoid is no longer cancellative in general. Regardless, one can use the general construction due to Grothendieck of an Abelian group from a commutative monoid. Applying this construction to the monoid of good sequences, we obtain a lattice Abelian group with strong units, which we shall refer to as the Chang's $\ell$-group of the BL-algebra. We also consider the case of BL-algebras of cancellative type, where the monoid of good sequences becomes cancellative and isomorphic to a subgroup of the positive cone of the Chang's group. This construction extends in a natural way the functor $\Xi$ introduced in \cite{CM} to a functor from the category of BL-algebras to that of lattice ordered groups with strong units. Some authors have already investigated possible extensions of the Mundici's functor to more general algebras \cite{EG, GT}. In the present case, the Mundici's equivalence is no longer an equivalence (nearly an adjoint pair), and the failure to be invertible provides a new understanding of the gap between MV-algebras and BL-algebras.\\
Finally, since the algebraic addition in BL-algebras does not behave nearly as nicely as the one in MV-algebras, we have been very cautious (probably over cautious at times) in proving the results most of the time with identical techniques as in \cite{CM}. But, there are also several instances where the techniques used for MV-algebras no longer work, and we have found completely different techniques.

 \section{Preliminaries}   
  A \textit{hoop} is an algebra $\textbf{A}=(A,\otimes, \to, 1)$ such that $(A,\otimes,1)$ is a commutative monoid and for all $x,y,z\in A$:\\
$$
\begin{aligned}
x\to x&=1\\
x\otimes(x\to y)&=y\otimes(y\to x)\\
x\to(y\to z)&=(x\otimes y)\to z
\end{aligned}
$$
It is known \cite[Prop. 2.1]{am} that any hoop $(A,\otimes, \to, 1)$ is a (natural) ordered residuated commutative monoid, where the order is defined by $x\leq y$ iff $x\to y=1$ and the residuation is: $$x\otimes y\leq z \; \text{iff}\; x\leq y\to z$$
In addition, a hoop is a lower semilattice order, where $x\wedge y=x\otimes(x\to y)$.\\
If the order of the hoop $A$ is linear, $A$ is called totally ordered hoop. A \textit{Wasjsberg hoop} is a hoop that satisfies the equation:$(x\to y)\to y=(y\to x)\to x$.\\  If the hoop $A$ has a minimum element $0$ with respect to the order above, then $A$ is called bounded hoop and is denoted by $(A,\to, \otimes, 0, 1)$.\\
A \textit{Wasjberg algebra} is a bounded A Wasjberg hoop.\\
A BL-algebra is a bounded hoop $(A,\to, \otimes, 0, 1)$ satisfying for all $x,y,z\in A$:\\
$(x\to y)\to z\leq ((y\to x)\to z)\to z$.
Every BL-algebra has the complementation operation defined by $\bar{x}=x\rightarrow 0$.\\
A product BL-algebra is a BL-algebra satisfying:$$x\wedge \overline{x}=0 \; \ \ \text{and}\; \ \ \overline{\overline{x}}\to ((x\otimes z\to y\otimes z)\to (x\to y) )=1$$
A complete study of product BL-algebras is found in \cite{C3}.\\
A G\"odel algebra (G-algebra) is a BL-algebra satisfying $x\otimes x=x$.\\
The following axioms of BL-algebras that shall be needed can be found  \cite{H0},\cite{H2},\cite{T1}, \cite{T2}.
\begin{align} 
& x\leq y \; \text{iff}\; x\rightarrow y=1; x\otimes y \leq x\wedge y;\\
& x\rightarrow(y\rightarrow z)=(x\otimes y)\rightarrow z;\\
& x\rightarrow (y\rightarrow z)=y\rightarrow (x\rightarrow z);\\
& \text{If}\; x\leq y, \; \text{then}\; y\rightarrow z\leq x\rightarrow z\; \text{and}\; z\rightarrow x\leq z\rightarrow y;\\
& x\leq y\rightarrow (x\otimes y);  x\otimes (x\rightarrow y)\leq y;\\
& 1\rightarrow x=x; x\rightarrow x=1; x\rightarrow 1=1; x\leq y\rightarrow x, x\leq \bar{\bar{x}}, \bar{\bar{\bar{x}}}=\bar{x};\\
& x\otimes \bar{x}= 0; x\otimes y=0\; \text{iff}\;   x\leq \bar{y};\\
& x\leq y\; \text{implies}\;  x\otimes z \leq y\otimes z,  z \rightarrow x \leq  z \rightarrow y, y \rightarrow z \leq  x \rightarrow z, \bar{y} \leq \bar{x} ;\\
& \overline{x\otimes y}  =  x \rightarrow \bar{y}; \overline{x \wedge y}  = \bar{x} \vee \bar{y},\overline{x \vee y}  = \bar{x} \wedge \bar{y}, \bar{0}=1 \; \text{and}\; \bar{1}=0;\\
& \overline{\overline{x\rightarrow y}}= \bar{\bar{x}}  \rightarrow  \bar{\bar{y}},  \overline{\overline{x \wedge y}}= \bar{\bar{x}}  \wedge   \bar{\bar{y}},  \overline{\overline{x\vee y}}= \bar{\bar{x}}   \vee    \bar{\bar{y}}, \bar{\bar{x}}\otimes\bar{\bar{y}}=\overline{\overline{x \otimes y}};\\
& x\otimes (y\vee z)= (x\otimes y)\vee (x\otimes z),  x\otimes (y\wedge z)= (x\otimes y)\wedge (x\otimes z);\\
& x\rightarrow (y\wedge z)= (x\rightarrow y)\wedge (x\rightarrow z); x\rightarrow (y\vee z)= (x\rightarrow y)\vee (x\rightarrow z).
\end{align}
An MV-algebra is a BL-algebra satisfying the double negation, that is $\bar{\bar{x}}=x$. It is known that this is equivalent to being a Wasjberg algebra.\\
In the literature, MV-algebras are also defined as Abelian monoids $(M,\oplus , 0)$ with an involution $^\ast:M\to M$ (i. e.; $(x^\ast)^\ast=x$ for all $x\in M$) satisfying the following axioms for all $x, y\in M$: $0^\ast\oplus x=0^\ast$; $(x^\ast\oplus y)^\ast\oplus y=(y^\ast\oplus x)^\ast\oplus x$.\\
 The two views are known to be equivalent. In fact, starting with a BL-algebra $M$ satisfying the double negation, if one writes $x\oplus y:=\bar{x}\to y$ and $x^\ast:=\bar{x}$, then $(M, \oplus, ^\ast, 0)$ satisfies the definition above. Conversely, given $(M, \oplus, ^\ast, 0)$ is an MV-algebra, if one defines $x\otimes y=(x^\ast\oplus y^\ast)^\ast$; $x\rightarrow y=x^\ast \oplus y$; $x\wedge y=x\otimes (y\oplus x^\ast)$; $x\vee y=x\oplus (y\otimes x^\ast)$; $x^\ast=\bar{x}$ and $1=\bar{0}$ where $\bar{x}=x\rightarrow 0$. Then $(M,\wedge, \vee, \otimes, \rightarrow, 0, 1)$ is a BL-algebra satisfying the double negation.\\
 Therefore in a Wasjberg algebra, $x\oplus y$ will denote the element $\bar{x}\to y$.\\
For any BL-algebra $L$, the subset  $ MV(L) =\{\bar{x}, x \in L \}$ is the largest MV-sub algebra of  $L$  and is called the MV-center of  $L$ \cite{T3}. The addition in the MV-center is defined by $\bar{x}\oplus \bar{y} = \overline{x \otimes y}$ for any $ \bar{x}, \bar{y} \in MV(L)$. A detailed treatment of the MV-center is found in \cite{T3}.\\ 
We recall that  for any subset  $X $ of a BL-algebra  $L$, 
 $ \overline{X}= \{ \bar{x},  x \in X \}$.\\
 Recall the definition of ordinal sums of hoops.\\
 Let $(I,\leq)$ be a totally ordered set. For each $i\in I$, let $\textbf{A}_i=(A_i,\otimes_i,\to _i, 1)$ be a hoop such that for every $i\ne j$, $A_i\cap A_j=\{1\}$. Then, the ordinal sum of the family $(\textbf{A}_i)_{i\in I}$ is the hoop $\bigoplus_{i\in I}\textbf{A}_i=(\cup_{i\in I}A_i, \otimes, \to, 1)$ where the operations $\otimes$ and $\to$ are defined by:
 $$
 \begin{aligned}
 x\otimes y =
\left\{\begin{array}{ll}
  x\otimes_i y & \ \ \mbox{if} \ \ x,y\in A_i;\\
 x&  \ \ \mbox{if} \ \ x\in A_i\setminus\{1\}, \; y\in A_j\; \mbox{and}\; i<j;\\
  y&  \ \ \mbox{if} \ \ y\in A_i\setminus\{1\}, \; x\in A_j\; \mbox{and}\; i<j.\\
\end{array}\right.\\
 x\to y =
\left\{\begin{array}{ll}
  x\to_i y & \ \ \mbox{if} \ \ x,y\in A_i;\\
 1&  \ \ \mbox{if} \ \ x\in A_i\setminus\{1\}, \; y\in A_j\; \mbox{and}\; i<j;\\
  y&  \ \ \mbox{if} \ \ y\in A_i, \; x\in A_j\; \mbox{and}\; i<j.\\
\end{array}\right.\\
 \end{aligned}
 $$
 A \textit{tower} of totally order Wasjberg hoops is a family $\tau=(C_i:i\in I)$ indexed by a totally ordered set $(I,\leq )$ with minimum $0$ such that for each $i\in I$
 $$\textbf{C}_i=(C_i,\otimes_i,\to _i, 1)$$
 is a totally ordered Wasjberg hoop, with $C_i\cap C_j=\{1\}$ for $i\ne j$ and $\textbf{C}_0$ is a Wasjberg algebra (MV-algebra). \\
 For every tower of totally order Wasjberg hoops is a family $\tau=(C_i:i\in I)$, $\textbf{A}_\tau$ denotes the ordinal sum of the family $\tau$.\\
 The following result which is found in  \cite[Theorem 3.7]{am}, or \cite[Theorem 3.4]{Bu} is the most important result on the structure of totally ordered BL-algebras (BL-chains).
 \begin{thm}\label{bl-chain}
 Every BL-chain $\textbf{A}$ is isomorphic to an algebra of the form $\textbf{A}_\tau$ for some tower $\tau$ of totally order Wasjberg hoops.
 \end{thm}
 We also recall the following representation theorem for BL-algebras.
 \begin{thm}\label{bl=s}\cite[Lemma 2.3.16]{H2}
 Every BL-algebra is a sub-direct product of BL-chains.
 \end{thm}
 \begin{rem}\label{equation}
 It follows from Theorem \ref{bl=s} that an equation holds in every BL-algebra if and only if its holds in every BL-chain.
 \end{rem}
  We also recall some basic facts about the algebraic addition treated in \cite{cn, ln}\\
 For every $x,y\in L $, we recall that $$x  \oslash y:= \bar{x}\rightarrow y $$ We can observe that if the BL-algebra is an MV-algebra, the operation $\oslash $ and $\oplus $ are the same. In particular, in the MV-center of $L $, $\oslash $ coincides with  $\oplus .$
\begin{lem}\label{pseudo} In every BL-algebra $L$.
\begin{itemize}
\item[(i).] The operation $\oslash $ is associative. That is for every $x,y, z \in L $,  $ (x  \oslash y) \oslash z = x  \oslash (y \oslash z)$. 
\item[(ii).] The operation $\oslash $ is compatible with the BL-order. That is for every $x,y, z, t \in L $,  such that   $ x \leq  y$ and $ z \leq t$, then  $ x\oslash z \leq  y\oslash t$.
\item[(iii).] The operation $\oslash$ distributes over $\vee$ and $\wedge$. That is for every $x,y,z\in L$, $ x  \oslash (y \vee z) = (x  \oslash  y)\vee (x \oslash z) $; $(x \vee y)  \oslash z= (x  \oslash  z)\vee (y \oslash z) $; $ x  \oslash (y \wedge z) = (x  \oslash  y)\wedge (x \oslash z) $; $(x \wedge y)  \oslash z= (x  \oslash  z)\wedge (y \oslash z) $
\item[(iv).] For every $x,y, z \in L $, $ x  \oslash (y \oslash z) = y  \oslash (x \oslash z) $. 
\end{itemize}
  \end{lem}
 Since the pseudo-addition $\oslash$ is not commutative and lacks an identity, we can construct a natural commutative operation $+$ as follows: 
 For every $x,y \in L $,  $$ x + y := ( x\oslash y)\wedge (y \oslash x)\ \ \ \ \ \ \ \ \ \ \ \ \ \ \ (\ast) $$
 Whenever $+$ is used in this work, it should be referring to this definition.\\
 The following lemma summarizes the main properties of the addition $+$.
 \begin{lem} In every BL-algebra $L$,
 \begin{itemize}
 \item[(i).] The operation $+ $ is associative. That is for every $x,y, z \in L $,  $ (x +  y)+  z = x  + (y + z) $. 
 \item[(ii).] The operation $+ $ is compatible with the BL-order. That is for every $x,y, z, t \in L $,  such that   $ x \leq  y$ and $ z \leq t$, then  $ x +z \leq  y + t $.
 \item[(iii).]  $0$ is the identity for $+$. That is, for every $x\in L$, $x+0=x$.
 \item[(iv).] $1$ is an absorbing element for $+$. That is $x+1=1$, for every $x\in L$.
\item[(v).]  For every $x,y\in L$, $x,y\leq x+y$.
\item[(vi).] For every $x\in L$, $\overline{\overline{x+\overline{x}}}=1$.
\item[(vii).] For every $x,y\in L$, $x+y=1$ implies $\bar{x}\leq y$.
\item[(viii).] For every $x, y\in L$, $\overline{\overline{x+y}}=\overline{\overline{x}}+\overline{\overline{y}}$.
\end{itemize}
 \end{lem}
 We shall also use some basic facts about lattice ordered groups, and for the convenience of the reader, we recall those facts here.\\
 A partially ordered group (po-group) is a group $(G,\cdot ,^{-1},1)$ equipped with a partial order $\leq$ that is translation invariant. In other words, for all $a,b,c\in G$, if $a\leq b$, then $a\cdot c\leq b\cdot c$. A subset $S$ of a po-group is called convex if for all $a,b,c\in G$ such that $a\leq b\leq c$ and $a,c\in S$, then $b\in S$. The positive cone of $G$, usually denoted by $G^{+}$ is defined as the subset of $G$ of all elements $x$ such that $0\leq x$. If $G$ is a po-group whose order defines a lattice structure, then $G$ is called a lattice-ordered group or $\ell$-group.  An $\ell$-subgroup of an $\ell$-group $G$ is a subgroup that is also a sublattice. An ideal of an $\ell$-group $G$ is a normal convex $\ell$-subgroup of $G$.\\
 The only $\ell$-groups that we shall deal with are Abelian, therefore we will use the additive notation $(G,+,-,0)$.
 Given an Abelian $\ell$-group $G$, an element $u\in G^{+}$ is called a strong unit if for all $x\in G$, there exists an integer $n\geq 1$, such that $x\leq nu$.\\
 We shall also use the following traditional notations: given $a\in G$, $a^+=a\vee 0$, $a^-=-a\vee 0$, in particular $a^+,a^-\in G^+$.\\
 The following result characterizes $\ell$-groups among all partially ordered abelian groups.
 \begin{prop}\label{lat-group}\cite[Prop. 3.3]{D}
 Let $(G,+,\leq, -, 0)$ be a partially ordered abelian group. Then $G$ is an $\ell$-group if and only for every $a\in G$, $a\vee 0$ exists in $G$.
 \end{prop}
 \begin{proof}
 In fact suppose $(G,+,\leq, -, 0)$ is partially ordered abelian group such that for every $a\in G$, $a\vee 0$ exists. Then for every $a, b\in G$, it is easy to see that $a\vee b$ exists and $a\vee b=((a-b)\vee 0)+b$. In addition $a\wedge b$ exists and $a\wedge b=-(-a\vee -b)$.
 \end{proof}
For details on lattice ordered groups, \cite{D} offers a complete treatment on the topic.
 \section{Addition and BL-chains}
We start by the following result provides a complete description of the addition in ordinal sums of tower of totally order Wasjberg hoops. 
 \begin{prop}\label{ad-ord}
Let $\tau=(C_i:i\in I)$ be a tower of totally order Wasjberg hoops. Then for every $x,y\in \textbf{A}_\tau$ with $x\ne 0$ and $y\ne 0$;
 $$
 \begin{aligned}
 x+ y =
\left\{\begin{array}{ll}
  x\oplus_0 y & \ \ \mbox{if} \ \ x,y\in C_0;\\
 1&   \ \  \mbox{otherwise}.\\
\end{array}\right.\\
 \end{aligned}
 $$
 \end{prop}
 \begin{proof}
 By the definition of the ordinal sum, a simple calculation shows that in $\textbf{A}_\tau$:
  $$
 \begin{aligned}
 \bar{x} =
\left\{\begin{array}{ll}
  \bar{x}^0 & \ \ \mbox{if} \ \ x\in C_0;\\
0&  \ \  \mbox{otherwise}.\\
\end{array}\right.\\
 \end{aligned}
 $$
where $\bar{x}^0:=x\to_00$. Now, the result follows from the definitions of addition and that of $\to$ in the ordinal sum.
 \end{proof}
 \begin{lem}\label{goodchain}
Let $L$ be a BL-chain, and $x, y\in L$.\
\begin{itemize}
\item[(i).] $x+y<1$ implies $x\otimes y=0$.
\item[(ii).] $x+y=x+z$ and $x\otimes y=x\otimes z$ imply $\bar{y}=\bar{z}$.
\item[(iii).] $x+y=x+z<1$ implies $\bar{y}=\bar{z}$.
\item[(iv).] $x+y=x$ implies $x=1$ or $y=0$.
\item[(v).] $x+y+(x\otimes y)=x+y$
\item[(vi)] $(x\otimes y)+((x+y)\otimes z)=(x+y)\otimes ((x\otimes y)+z)$.
\item[(vii).] $(x\vee y)+z=(x+z)\vee (y+z)$ and $(x\wedge y)+z=(x+z)\wedge (y+z)$
\item[(viii).] $x+y=x$ if and only if $\bar{x}+\bar{y}=\bar{y}$.
\end{itemize}
\end{lem}
\begin{proof}
We could use the decomposition of BL-chains given by Theorem \ref{bl-chain}, Proposition \ref{ad-ord}, combined with similar properties for MV-chains\cite[Lemma 1.6.1]{C2}. But since direct proofs are equally simple, we prefer to give direct proofs.\\
(i). Since $x+y<1$, then $x\oslash y\ne 1$ or $y\oslash x \ne 1$, that is $\bar{x}\nleq y$ or $\bar{y}\nleq x$. Since $L$ is chain, then $y\leq \bar{x}$ or $x\leq \bar{y}$. But in each case, $x\otimes y=0$ as needed.\\
(ii). Suppose that $x+y=x+z$ and $x\otimes y=x\otimes z$, then $\overline{\overline{x}}+\overline{\overline{y}}=\overline{\overline{x}}+\overline{\overline{z}}$ and $\overline{\overline{x}}\otimes \overline{\overline{y}}=\overline{\overline{x}}\otimes \overline{\overline{z}}$. It follows from a similar result for MV-chains that $\overline{\overline{y}}=\overline{\overline{z}}$, from which we obtain $\bar{y}=\bar{z}$.\\
(iii). Suppose that $x+y=x+z<1$, then by (i), $x\otimes y=x\otimes z=0$. Hence, the result follows from (ii).\\
(iv). Suppose that $x+y=x$ and $x<1$, then $x+y=x+0<1$ and it follows from (iii) that $\bar{y}=1$. Thus, $y=0$ as required.\\
(v). If $x\otimes y=0$, the result is clear. On the other hand, if $x\otimes y\ne 0$, then by (i), $x+y=1$ and the equation is also obvious.\\
(vi) If $x+y=1$, then both sides of the equation equal $(x\otimes y)+z$. But, if $x+y<1$, then by (i) both sides of the equation equal $(x+y)\otimes z$.\\
(vii) If $x\leq y$, then $x+z\leq y+z$ and $(x\vee y)+z=(x+z)\vee (y+z)=y+z$. The proof when $y\leq x$ is similar.\\
(viii) This follows easily from (iv).
\end{proof}
\begin{rem}\label{sub}
It follows from Remark \ref{equation} that the equations (v), (vi) and (vii) of Lemma \ref{goodchain} above hold in every BL-algebra. We should also point out that in property (ii), the conclusion cannot be replaced by $y=z$.  
\end{rem}
Recall \cite{am} that a hoop $\textbf{A}=(A,\otimes, \to, 1)$ is called cancellative if $(A,\otimes)$ is a cancellative monoid. It is therefore clear that no nontrivial BL-algebra is cancellative since $0\otimes 1=0$. For BL-algebras, we introduce the following definition.
\begin{defi}
A BL-chain $L$ is called \textit{of cancellative type} if for every $x,y,z\in L$,
\begin{center}
$x+y=x+z$ and $x\otimes y=x\otimes z$ imply $y=z$.
\end{center}
A BL-algebra is called of cancellative type if it is a subdirect product of BL-chains of cancellative type.
\end{defi}
Note that if $L$ is a BL-chain such that $(L\setminus \{0\}, \otimes, \to, 1)$ is a cancellative hoop, then $L$ is of cancellative type.
\begin{ex}
(i) Every MV-algebra is of cancellative type.\\
(ii)The Product structure on $[0,1]$ is of cancellative type, while the G\"{o}del structure on $[0,1]$ is not.\\
(iii) More generally, every product BL-algebra is of cancellative type. In fact, it follows from \cite[Lemma 5.1]{afm} that every product BL-chain is of cancellative type. It is also clear that every product BL-algebra is a subdirect product of product BL-chains.
\end{ex}
The following Lemma is key for the rest of the paper.
\begin{lem}\label{Idass}
In every BL-algebra $L$, the following equation holds.
$$(x\otimes y)+((x+y)\otimes z)=(x\otimes z)+((x+z)\otimes y)$$
\end{lem}
\begin{proof}
Again, by Remark \ref{equation}, there is no harm in assuming that $L$ is a BL-chain. But, by Theorem \ref{bl-chain}, there exists a tower of totally order Wasjberg hoops $\tau=(C_i:i\in I)$ such that $L\cong \textbf{A}_\tau$. We shall justify the equation for $\textbf{A}_\tau$ by considering cases.\\
First, it is clear that the equation holds if any of $x,y,z$ is equal to $0$ or $1$, so we may assume that $\{x,y,z\}\cap\{0,1\}=\emptyset$.\\
It is also clear from the definition of $\otimes$ in the ordinal sum that for $x\in C_0\setminus\{1\}$ and $y\notin C_0$, then $x\otimes y=x$.\\
Case 1: If $x,y,z\in C_0$, the equation holds since $C_0$ is an MV-chain MV-chains\cite[Lemma 1.6.1]{C2}.\\
Case 2: If $x,y\in C_0$ and $z\notin C_0$. Then, $(x\otimes y)+((x+y)\otimes z)=(x\otimes _0y)\oplus_0(x\oplus_0y)=x\oplus_0y=(x\otimes z)+y=(x\otimes z)+((x+z)\otimes y)$.\\
Case 3: If $x,z\in C_0$ and $y\notin C_0$. Then, $(x\otimes y)+((x+y)\otimes z)=x\oplus_0z=(x\otimes _0z)\oplus_0(x\oplus_0z)=(x\otimes z)+((x+z)\otimes y)$ \\
Case 4: If $y,z\in C_0$ and $x\notin C_0$. Then, $(x\otimes y)+((x+y)\otimes z)=y\oplus_0z=z\oplus_0 x=(x\otimes z)+((x+z)\otimes y)$.\\
Case 5: If $x\in C_0$ and $y, z\notin C_0$. Then, both sides of the equation are equal to $1$.\\
Case 6: If $y\in C_0$ and $x,z\notin C_0$. Then $(x\otimes y)+((x+y)\otimes z)=1$ and $(x\otimes z)+((x+z)\otimes y)=(x\otimes z)+y$. If $x\otimes z\notin C_0$, then $(x\otimes z)+y=1$. If $x\otimes z\in C_0$, then $\overline{x\otimes z}=\overline{x\otimes z}^0\in C_0$. But, $\overline{x\otimes z}=x\to \bar{z}=x\to 0=\bar{x}=0$. Hence, $x\otimes z=1$ since $C_0$ is an MV-algebra. Therefore, both sides of the equation are equal to $1$.\\
Case 7: If $z\in C_0$ and $x,y\notin C_0$. This case is similar to case 6.\\
Case 8: If $x,y,z\notin C_0$. An analysis similar to the one in case 6 combined with Proposition \ref{ad-ord} shows that both sides of the equation are equal to $1$.
\end{proof}
 \section{Good Sequences in BL-algebras}
\begin{defi}
Let $L$ be a BL-algebra. A sequence $\mathbf{a}:=(a_1, a_2, \ldots, a_n, \ldots)$ of elements of $L$ is called a good sequence if for all $i$, $a_{i}+a_{i+1}=a_i$, and there exists an integer $n$ such that $a_r=0$ for all $r>n$. In this case, instead of writing $\mathbf{a}=(a_1, a_2, \ldots, a_n, 0, 0 \ldots)$, we will simply write $\mathbf{a}=(a_1, a_2, \ldots, a_n)$
\end{defi} 
Note that if, $\mathbf{a}=(a_1, a_2, \ldots, a_n)$ is a good sequence, then $(\underbrace{1, \ldots, 1}_{m\;'s\; 1},a_1, a_2, \ldots, a_n)$ is again a good sequence. The later will be denoted by $(1^m,a_1, a_2, \cdots, a_n)$.\\
We define the addition of sequences as follows:
\begin{defi}
Let $L$ be a BL-algebra and let $\mathbf{a}=(a_1, a_2, \ldots, a_n, \ldots)$, $\mathbf{b}=(b_1, b_2, \ldots, b_n, \ldots)$ be sequences in $L$. We define the sum of $\mathbf{a}$ and $\mathbf{b}$ by $\mathbf{a}+\mathbf{b}=(c_1,c_2,\ldots, c_n, \ldots)$ where $c_i:=a_i+(a_{i-1}\otimes b_1)+\cdots +(a_1\otimes b_{i-1})+b_i$.
\end{defi}
It is clear that the addition above is commutative as $+$ and $\otimes$ are commutative in $L$.
We start by characterizing good sequences of BL-chains. 
\begin{prop}\label{goodschain}
Let $L$ be a BL-chain, then good sequences of $L$ are of the form $(1^p,a)$ for some integer $p\geq 0$ and $a\in L$.
\end{prop}
\begin{proof}
This is clear from Lemma \ref{goodchain}(iv).
\end{proof}
\begin{rem}\label{sumc}
If $C$ is a BL-chain and $\mathbf{a}, \mathbf{b}$ are two good sequences in $C$, then by Proposition \ref{goodschain}, there exists integers $p,q\geq 0$ and $a, b\in L$ such that $\mathbf{a}=(1^p,a)$ and $\mathbf{b}=(1^q,b)$. Therefore, $\mathbf{a}+\mathbf{b}=(1^{p+q}, a+b, a\otimes b)$ and it follows from Lemma \ref{goodchain} (v) that $\mathbf{a}+\mathbf{b}$ is a good sequence. Thus, the sum of good sequences in a BL-chain is again a good sequence.
\end{rem}
\begin{lem}\label{goodc}
Suppose that $L\subseteq \prod_iC_i$ be a subdirect product of BL-chains and $\pi_i:\prod_iC_i\to C_i$ is the natural projection. If $\mathbf{a}=(a_1, a_2, \ldots, a_n, \ldots)$ is a sequence in $L$, for every $i$, let $\mathbf{a}_i=(\pi_i(a_1), \pi_i(a_2), \ldots, \pi_i(a_n), \ldots)$ which is a sequence in $C_i$. Then;\\
1. A sequence $\mathbf{a}$ is a good sequence in $L$ if and only if for every $i$, $\mathbf{a}_i$ is a good sequence in $C_i$ and there exists $m\geq 0$ such that $\pi_i(a_n)=0$ for all $n>m$ and for all $i$.\\
2. For every good sequences $\mathbf{a}, \mathbf{b}$ in $L$, $\mathbf{a}+\mathbf{b}$ is a good sequence in $L$ if and only if for all $i$, $\mathbf{a}_i+\mathbf{b}_i$ is a good sequence in $C_i$.
\end{lem}
\begin{proof}
1. This follows from the fact that $a_n+a_{n+1}=a_n$ if and only if $\pi_i(a_n+a_{n+1})=\pi_i(a_n)$ for all $i$.\\
2.  The proof is simple and follows once again from the fact that every $\pi_i$ is a BL-homomorphism, and $x=y$ in $L$ if and only if $\pi_i(x)=\pi_i(y)$ for all $i$.
\end{proof}
It follows from Lemma \ref{goodc} (2) and Remark \ref{sumc} that in any BL-algebra, the sum of two good sequences in again a good sequences. It is clear that for every good sequence $\mathbf{a}$ in $L$, $\mathbf{a}+(0)=\mathbf{a}$.\\
The following Lemma whose proof is straightforward reduces the associativity of the addition of good sequences to the case of BL-chains.
\begin{lem}\label{goodass}
Suppose that $L\subseteq \prod_iC_i$ be a subdirect product of BL-chains and $\pi_i:\prod_iC_i\to C_i$ is the natural projection. If $\mathbf{a}=(a_1, a_2, \ldots, a_n, \ldots)$ is a sequence in $L$, for every $i$, let $\mathbf{a}_i=(\pi_i(a_1), \pi_i(a_2), \ldots, \pi_i(a_n), \ldots)$ which is a sequence in $C_i$. Then for every sequences $\mathbf{a}, \mathbf{b}, \mathbf{c}$ in $L$, $(\mathbf{a}+ \mathbf{b})+\mathbf{c}=\mathbf{a}+(\mathbf{b}+\mathbf{c})$ if and only if for all $i$, $(\mathbf{a}_i+ \mathbf{b}_i)+\mathbf{c}_i=\mathbf{a}_i+(\mathbf{b}_i+\mathbf{c}_i)$.
\end{lem}
Throughout the rest of this paper, $M_L$ will denote the set of all good sequences in the BL-algebra $L$.
\begin{prop}\label{assoc}
For every BL-algebra $L$, $(M_L, +, (0))$ is an Abelian monoid satisfying: for every $\mathbf{a}, \mathbf{b}\in M_L$,
$$\mathbf{a}+\mathbf{b}=(0)\ \ \; \text{implies}\; \  \ \mathbf{a}=\mathbf{b}=(0) $$
\end{prop}
\begin{proof}
It remains to prove that the addition of good sequences is associative. By Lemma \ref{goodass}, it is enough to show this for BL-chains. Let $C$ be a BL-chain, and let $\mathbf{a}, \mathbf{b}, \mathbf{c}$ in $C$, then by Lemma \ref{goodschain}, there exist integers $p,q,r\geq 0$ and $a,b,c\in C$ such that $\mathbf{a}=(1^p,a)$, $\mathbf{b}=(1^q,b)$, $\mathbf{c}=(1^r,c)$. Using the identies of Lemma \ref{goodchain}(v) and Lemma \ref{Idass}, we have:
$$
\begin{aligned}
(\mathbf{a}+ \mathbf{b})+\mathbf{c}&=(1^{p+q},a+b,a\otimes b)+(1^r,c)\\
&=(1^{p+q+r}, a\otimes b+a+b+c, a\otimes b+(a+b)\otimes c, a\otimes b\otimes c)\\
&=(1^{p+q+r}, a+b+c, a\otimes (b+c)+(b\otimes c), a+b+c+b\otimes c)\\
&=(1^{p},a)+(1^{q+r}, b+c, b\otimes c)\\
&=\mathbf{a}+(\mathbf{b}+\mathbf{c})
\end{aligned}
$$
Therefore, the addition $+$ is associative in $M_L$.\\
On the other hand, if $\mathbf{a}+\mathbf{b}=(0)$, by definition of the addition for all $i$, $a_i, b_i\leq (\mathbf{a}+\mathbf{b})_i=0$. Hence $a_i=b_i=0$ for all $i$.
\end{proof}
\section{The Chang's group of a BL-algebra}
Using the Grothendieck construction of a group from a commutative monoid, we define on $M_L\times M_L$ the relation $\sim$ by $(\mathbf{a}, \mathbf{b})\sim (\mathbf{c}, \mathbf{d})$ if and only if there exists $\mathbf{k}\in M_L$ such that $\mathbf{a}+\mathbf{d}+\mathbf{k}=\mathbf{b}+\mathbf{c}+\mathbf{k}$. Then, it is easy to see that $\sim$ is an equivalence relation on $M_L\times M_L$. Let $G_L=M_L\times M_L/\sim$ the factor set of $M_L\times M_L$ by $\sim$. If we denote the equivalence class of $(\mathbf{a}, \mathbf{b})$ by $[\mathbf{a}, \mathbf{b}]$, one has a well-defined operation $+$ on $G_L$ given by $[\mathbf{a}, \mathbf{b}]+[\mathbf{c}, \mathbf{d}]=[\mathbf{a+c}, \mathbf{b+d}]$. Moreover, $(G_L, +, [(0), (0)])$ is an Abelian group where $-[\mathbf{a}, \mathbf{b}]=[\mathbf{b}, \mathbf{a}]$. We shall refer to this group as the Chang's group of the BL-algebra $L$. This terminology is motivated by the fact that when $L$ is an MV-algebra, this group coincides with the well-known Chang's $\ell$-group of an MV-algebra.\\
Our next goal is to add an order structure on $G_L$.\\
We start with the order $\leq$ on the monoid $(M_L,+)$ defined by:  $\mathbf{a}, \mathbf{b}\in M_L$; $\mathbf{a}\leq \mathbf{b}$ if $a_i\leq b_i$ for all $i$. It is clear that $\leq$ is a partial order on $M_L$. As the next result show, $\leq$ is actually a lattice order. 
\begin{prop}
$(M_L,\leq)$ is a lattice where for every $\mathbf{a}, \mathbf{b}\in M_L$:\\
1. $\mathbf{a}\vee \mathbf{b}=(a_1\vee b_1, a_2\vee b_2,\ldots, a_n\vee b_n,\ldots)$ and\\
2. $\mathbf{a}\wedge \mathbf{b}=(a_1\wedge b_1, a_2\wedge b_2,\ldots, a_n\wedge b_n,\ldots)$.
\end{prop}
\begin{proof}
Note that it is enough to prove that if $\mathbf{a}, \mathbf{b}$ are good sequences in $L$, so are $(a_1\vee b_1, a_2\vee b_2,\ldots, a_n\vee b_n,\ldots)$ and $(a_1\wedge b_1, a_2\wedge b_2,\ldots, a_n\wedge b_n,\ldots)$. This is because, since the order $\leq$ is coordinates-wise, it would therefore automatic that the formulae provide the supremum and the infimum.\\
1. Because of Lemma \ref{goodc}, we may assume that $L$ is a BL-chain. Let $\mathbf{a}, \mathbf{b}$ be good sequences in $L$, then $\mathbf{a}=(1^p,a)$ and $\mathbf{b}=(1^q,b)$ for some integers $p,q\geq 0$ and $a,b\in L$. Then,
 $$
 \begin{aligned}
(a_1\vee b_1, a_2\vee b_2,\ldots, a_n\vee b_n,\ldots) =
\left\{\begin{array}{ll}
  (1^q,b) & \ \ \mbox{if} \ \ p<q;\\
 (1^p,a)&  \ \ \mbox{if} \ \ p>q;\\
  (1^p,a\vee b)& \ \ \mbox{if} \ \ p=q.\\
\end{array}\right.\\
  \end{aligned}
 $$
The resulting sequence is clearly a good sequence.\\
2. The proof is similar to that of 1.
\end{proof}
We have the following important property of the monoid $(M_L,+)$.
\begin{prop}\label{plus-dist}
Let $L$ be a BL-algebra and $\mathbf{a}, \mathbf{b}, \mathbf{c}$ be good sequences in $L$, then:
$$(\mathbf{a}\vee \mathbf{b})+ \mathbf{c}=(\mathbf{a}+\mathbf{c})\vee (\mathbf{b}+\mathbf{c})\; \ \ \text{and}\ \ \; (\mathbf{a}\wedge \mathbf{b})+ \mathbf{c}=(\mathbf{a}+\mathbf{c})\wedge (\mathbf{b}+\mathbf{c})$$
\end{prop}
\begin{proof}
Again, thanks to Lemma \ref{goodc}, we may assume that $L$ is a BL-chain. In this case, there exist integers $p,q,r\geq 0$ such that $\mathbf{a}=(1^p,a)$, $\mathbf{b}=(1^q,b)$ and $\mathbf{c}=(1^r,c)$. As above, 
 $$
 \begin{aligned}
\mathbf{a}\vee \mathbf{b} =
\left\{\begin{array}{ll}
  (1^q,b) & \ \ \mbox{if} \ \ p<q;\\
 (1^p,a)&  \ \ \mbox{if} \ \ p>q;\\
  (1^p,a\vee b)& \ \ \mbox{if} \ \ p=q.\\
\end{array}\right.\\
  \end{aligned}
 $$
Now, we consider the following cases.\\
Case 1: If $p<q$, then $(\mathbf{a}\vee \mathbf{b})+ \mathbf{c}=(1^{q+r}, b+x, b\otimes x)$.\\
Subcase 1.1: If $p<q-1$, then $(\mathbf{a}+\mathbf{c})\vee (\mathbf{b}+\mathbf{c})=(1^{q+r}, b+x, b\otimes x)$ and the equation is clear.\\
Subcase 1.2: If $p=q-1$, then $(\mathbf{a}+\mathbf{c})\vee (\mathbf{b}+\mathbf{c})=(1^{q+r}, (a\otimes x)\vee(b+x), b\otimes x)=(1^{q+r}, b+x, b\otimes x)$ since $a\otimes x\leq b+x$. Therefore, the equality is again verified.\\
Case 2: If $p=q$, then $(\mathbf{a}\vee \mathbf{b})+ \mathbf{c}=(1^{p+r}, (a\vee b)+x, (a\vee b)\otimes x)=(1^{p+r}, (a+x)\vee (b+x), (a\otimes x)\vee (b\otimes x))$. The last equality is due to Lemma \ref{goodchain} (vii). The equation holds because $(1^{p+r}, (a+x)\vee (b+x), (a\otimes x)\vee (b\otimes x))=(\mathbf{a}+\mathbf{c})\vee (\mathbf{b}+\mathbf{c})$.\\
Case 3: The case $p>q$ is symmetric to case 1.\\
The proof of the second equation is similar to the above.
\end{proof}
We can consider the order on $G_L$ as follows. 
\begin{prop}
Let $\preceq$ be the relation defined on $G_L$ by $[\mathbf{a}, \mathbf{b}]\preceq [\mathbf{c}, \mathbf{d}]$ if and only if there exists $\mathbf{k}\in M_L$ such that $\mathbf{a}+\mathbf{d}+\mathbf{k}\leq \mathbf{b}+\mathbf{c}+\mathbf{k}$.\\
Then $\preceq$ is a partial order relation on $G_L$ that is translation invariant.
\end{prop}
\begin{proof}
The proof is a routine verification that depends only on the construction of the group from any commutative monoid.
\end{proof}
The partial order $\preceq$ induces a lattice structure on $G_L$ as we now prove.
\begin{thm}\label{lattice-o}
For every BL-algebra $L$, the partially ordered group $G_L$ is a lattice ordered group where,
$$[\mathbf{a},\mathbf{b}]\vee [\mathbf{c},\mathbf{d}]=[(\mathbf{a}+\mathbf{d})\vee(\mathbf{b}+\mathbf{c}), \mathbf{b}+\mathbf{d}]\ \ \; \text{and}\ \ \; [\mathbf{a},\mathbf{b}]\wedge [\mathbf{c},\mathbf{d}]=[(\mathbf{a}+\mathbf{d})\wedge(\mathbf{b}+\mathbf{c}), \mathbf{b}+\mathbf{d}]$$
\end{thm}
\begin{proof}
To show that $G_L$ is a lattice ordered group, it is enough by Proposition \ref{lat-group} to show that for every $\mathbf{a}, \mathbf{b}\in M_L$, $[\mathbf{a},\mathbf{b}]\vee [(0),(0)]$ exists in $(G_L, \preceq)$. We claim that $[\mathbf{a},\mathbf{b}]\vee [(0),(0)]=[\mathbf{a}\vee \mathbf{b}, \mathbf{b}]$. First, that $[\mathbf{a}\vee \mathbf{b}, \mathbf{b}]$ is an upper bound of $[\mathbf{a},\mathbf{b}]$ and $[(0),(0)]$ is simply because $\mathbf{b}\leq \mathbf{a}\vee \mathbf{b}$ and $\mathbf{a}+\mathbf{b}\leq (\mathbf{a}\vee\mathbf{b})+\mathbf{b}$. On the other hand, suppose that $[\mathbf{u}, \mathbf{v}]$ is any upper bound of $[\mathbf{a},\mathbf{b}]$ and $[(0),(0)]$. Then, there exist $\mathbf{k}_1, \mathbf{k}_2\in M_L$ such that $\mathbf{a}+\mathbf{v}+\mathbf{k}_1\leq \mathbf{b}+\mathbf{u}+\mathbf{k}_1$ and $\mathbf{v}+\mathbf{k}_2\leq \mathbf{u}+\mathbf{k}_2$. Hence, $\mathbf{a}+\mathbf{v}+\mathbf{k}_1+\mathbf{k}_2\leq \mathbf{b}+\mathbf{u}+\mathbf{k}_1+\mathbf{k}_2$ and $\mathbf{b}+\mathbf{v}+\mathbf{k}_1+\mathbf{k}_2\leq \mathbf{b}+ \mathbf{u}+\mathbf{k}_1+\mathbf{k}_2$. Thus, $(\mathbf{a}+\mathbf{v}+\mathbf{k}_1+\mathbf{k}_2)\vee (\mathbf{b}+\mathbf{v}+\mathbf{k}_1+\mathbf{k}_2)\leq \mathbf{b}+ \mathbf{u}+\mathbf{k}_1+\mathbf{k}_2$. But, by Proposition \ref{plus-dist} $(\mathbf{a}+\mathbf{v}+\mathbf{k}_1+\mathbf{k}_2)\vee (\mathbf{b}+\mathbf{v}+\mathbf{k}_1+\mathbf{k}_2)=(\mathbf{a}\vee \mathbf{b})+\mathbf{v}+\mathbf{k}_1+\mathbf{k}_2)\leq \mathbf{b}+ \mathbf{u}+\mathbf{k}_1+\mathbf{k}_2$. Therefore, $[\mathbf{a}\vee \mathbf{b}, \mathbf{b}]\preceq [\mathbf{u}, \mathbf{v}]$ as needed. \\
Finally, if $[\mathbf{a},\mathbf{b}], [\mathbf{c},\mathbf{d}]\in G_L$, combining the preceding special case and the argument in the proof of Proposition \ref{lat-group} yield, 
\[
\begin{aligned}
\left[\mathbf{a},\mathbf{b}\right] \vee [\mathbf{c},\mathbf{d}]&=(([\mathbf{a},\mathbf{b}]+ [\mathbf{d},\mathbf{c}])\vee [(0),(0)])+[\mathbf{c}, \mathbf{d}]\\
&=[(\mathbf{a}+\mathbf{d})\vee (\mathbf{b}+\mathbf{c}),\mathbf{b}+\mathbf{c}]+[\mathbf{c},\mathbf{d}]\\
&=[((\mathbf{a}+\mathbf{d})\vee (\mathbf{b}+\mathbf{c}))+\mathbf{c},\mathbf{b}+\mathbf{c}+\mathbf{d}]\\
&=[((\mathbf{a}+\mathbf{d})\vee (\mathbf{b}+\mathbf{c})),\mathbf{b}+\mathbf{d}]
\end{aligned}
\]
The proof for the greatest lower bound is similar and we leave it as an exercise. 
\end{proof}
\begin{rem}
The standard proof of Theorem \ref{lattice-o} for MV-algebras relies on the fact that if $A$ is an MV-algebra and $\mathbf{a}\leq \mathbf{b}$ in $M_A$ if and only if there exists $\mathbf{c}\in M_A$ such that $\mathbf{b}=\mathbf{a}+\mathbf{c}$ \cite[Prop. 2.3.2]{C2}. However, this condition does not hold for general BL-algebras, not even for those of cancellative type. For instance, when $L=[0,1]$ is equipped with the product structure, there does not exists a good sequence $\mathbf{a}$ in $L$ such that $(.5)=(.25)+\mathbf{a}$.\\
\end{rem}
\begin{prop}\label{strong}
For every BL-algebra $L$, $u_L:=[(1),(0)]$ is a strong unit of the $\ell$-group $G_L$.
\end{prop}
\begin{proof}
Let $[\mathbf{a},\mathbf{b}]\in G_L$, and let $n\geq 1$ be any integer such that $a_i=b_i=0$ for all $i>n$. Then, $n(1)=(1^n)$, $\mathbf{b}+(1^n)=(1^n)$, and $\mathbf{a}\leq (1^n)$. Thus, $[\mathbf{a},\mathbf{b}]\preceq nu_L$ as required.
\end{proof}
The Chang's $\ell$-group of the MV-center of a BL-algebra $L$ seats in a particularly nice way in $G_L$.
\begin{thm}\label{summand}
Let $L$ be a BL-algebra and let $A$ be the MV-center of $L$. Then, the Chang's $\ell$-group $G_A$ of $A$ is a direct summand of $G_L$. Moreover the complement of $G_A$ is an ideal of $G_L$.
\end{thm}
\begin{proof}
We shall need the following definition and notation. Given a good sequence $\mathbf{a}:=(a_1, a_2, \ldots,a_n)$ in $L$, let $\overline{\mathbf{a}}=(\overline{a_n}, \overline{a_{n-1}}, \ldots, \overline{a_1})$, in particular $\overline{\overline{\mathbf{a}}}=(\overline{\overline{a_1}}, \overline{\overline{a_2}}, \ldots, \overline{\overline{a_n}})$. It is clear from Lemma \ref{goodchain}(viii) that since $\mathbf{a}$ is a good sequence in $L$, so is $\overline{\mathbf{a}}$. Note that with these notations, $G_A=\{[\overline{\mathbf{a}},\overline{\mathbf{b}}]:\mathbf{a}, \mathbf{b}\in M_L\}$. Now, consider $\mathfrak{S}(L):=\{[\mathbf{a}, \mathbf{b}]: \overline{\overline{\mathbf{a}}}=\overline{\overline{\mathbf{b}}}\}$. We shall justify that $\mathfrak{S}(L)$ is a sub-lattice ordered group of $G_L$ and $G_L=G_A\oplus \mathfrak{S}(L)$.\\
One can easily verify the following equations given any $\mathbf{a}, \mathbf{b}$ in $M_L$.
$$
\begin{aligned}
\overline{\overline{\mathbf{a}+\mathbf{b}}}&=\overline{\overline{\mathbf{a}}}+\overline{\overline{\mathbf{b}}}\\
\overline{\overline{\mathbf{a}\vee\mathbf{b}}}&=\overline{\overline{\mathbf{a}}}\vee\overline{\overline{\mathbf{b}}}\\
\overline{\overline{\mathbf{a}\wedge\mathbf{b}}}&=\overline{\overline{\mathbf{a}}}\wedge\overline{\overline{\mathbf{b}}}
\end{aligned}
$$
A simple computation using the equations above shows that $\mathfrak{S}(L)$ is a subgroup of $G_L$. In addition, one can also verify that using the formulas in Theorem \ref{lattice-o}, if $[\mathbf{a},\mathbf{b}], [\mathbf{c},\mathbf{d}]\in \mathfrak{S}(L)$, then $[\mathbf{a},\mathbf{b}]\vee [\mathbf{c},\mathbf{d}]\in \mathfrak{S}(L)$ and $[\mathbf{a},\mathbf{b}]\wedge [\mathbf{c},\mathbf{d}]\in \mathfrak{S}(L)$.\\
 So, $\mathfrak{S}(L)$ is a sub-lattice ordered group of $G_L$. On the other hand, every element $[\mathbf{a},\mathbf{b}]$  can be decomposed as $[\mathbf{a},\mathbf{b}]=[\mathbf{a}+\overline{\overline{\mathbf{b}}},\mathbf{b}+\overline{\overline{\mathbf{a}}}]+[\overline{\overline{\mathbf{a}}},\overline{\overline{\mathbf{b}}}]$. On one hand, it clear that $[\overline{\overline{\mathbf{a}}},\overline{\overline{\mathbf{b}}}]\in G_A$ and on the other hand, it follows from the equations above and the identity $\overline{\overline{\overline{\mathbf{a}}}}= \overline{\mathbf{a}}$ that $[\mathbf{a}+\overline{\overline{\mathbf{b}}},\mathbf{b}+\overline{\overline{\mathbf{a}}}]\in \mathfrak{S}(L)$. Finally, if $[\overline{\mathbf{a}},\overline{\mathbf{b}}]\in \mathfrak{S}(L)$, then $\overline{\mathbf{a}}=\overline{\overline{\overline{\mathbf{a}}}}=\overline{\overline{\overline{\mathbf{b}}}}=\overline{\mathbf{b}}$. Therefore, $[\overline{\mathbf{a}},\overline{\mathbf{b}}]=[(0),(0)]$ and $G_A\cap \mathfrak{S}(L)=\{[(0),(0)]\}$ as required.\\
It remains to show that $\mathfrak{S}(L)$ is an ideal of $G_L$, or equivalently that $\mathfrak{S}(L)$ is convex. For this, let $[\mathbf{a},\mathbf{b}], [\mathbf{c},\mathbf{d}]\in \mathfrak{S}(L)$ and $[\mathbf{e},\mathbf{f}]\in G_L$ such that $[\mathbf{a},\mathbf{b}]\preceq [\mathbf{e},\mathbf{f}]\preceq [\mathbf{c},\mathbf{d}]$. Then there exists $\mathbf{k}_1, \mathbf{k}_2\in M_L$ such that $\mathbf{a}+\mathbf{f}+\mathbf{k}_1\leq \mathbf{b}+\mathbf{e}+\mathbf{k}_1$ and $\mathbf{e}+\mathbf{d}+\mathbf{k}_2\leq \mathbf{c}+\mathbf{f}+\mathbf{k}_2$. Hence, $\overline{\overline{\mathbf{a}}}+\overline{\overline{\mathbf{f}}}+\overline{\overline{\mathbf{k}_1}}\leq \overline{\overline{\mathbf{b}}}+\overline{\overline{\mathbf{e}}}+\overline{\overline{\mathbf{k}_1}}$ and $\overline{\overline{\mathbf{e}}}+\overline{\overline{\mathbf{d}}}+\overline{\overline{\mathbf{k}_2}}\leq \overline{\overline{\mathbf{c}}}+\overline{\overline{\mathbf{f}}}+\overline{\overline{\mathbf{k}_2}}$. Since $\overline{\overline{\mathbf{a}}}=\overline{\overline{\mathbf{b}}}$ and $\overline{\overline{\mathbf{c}}}=\overline{\overline{\mathbf{d}}}$, it follows that $\overline{\overline{\mathbf{e}}}+ \overline{\overline{\mathbf{u}}}\leq \overline{\overline{\mathbf{f}}}+ \overline{\overline{\mathbf{u}}}$ and $\overline{\overline{\mathbf{f}}}+ \overline{\overline{\mathbf{u}}}\leq \overline{\overline{\mathbf{e}}}+ \overline{\overline{\mathbf{u}}}$, where $\mathbf{u}=\mathbf{a}+\mathbf{c}+\mathbf{k}_1+\mathbf{k}_2$. Hence, $\overline{\overline{\mathbf{e}}}+ \overline{\overline{\mathbf{u}}}=\overline{\overline{\mathbf{f}}}+ \overline{\overline{\mathbf{u}}}$, and since $M_A$ is cancellative \cite[Prop. 2.1]{CM}, then $\overline{\overline{\mathbf{e}}}=\overline{\overline{\mathbf{f}}}$ and $[\mathbf{e},\mathbf{f}]\in \mathfrak{S}(L)$. Therefore, $\mathfrak{S}(L)$ is an ideal of $G_L$ as needed.
\end{proof}
It is worth pointing out that unlike $\mathfrak{S}(L)$, $G_A$ is not an ideal of $G_L$ in general. For example, if one considers the BL-chain $L=[0,1]$ with the product structure, which has the two-element Boolean algebra $\mathbf{2}$ as MV-center. Then it is clear that $[(0),(0)]\preceq [(.5),(0)]\preceq [(1),(0)]$ and $[(.5),(0)]\notin G_{\mathbf{2}}$.\\
Note that if $L$ a BL-algebra that is an MV-algebra, then $\mathfrak{S}(L)=0$, and it follows from Theorem \ref{summand} that $G_L$ coincides with the well known Chang's $\ell$-group of an MV-algebra.
\begin{cor}\label{isom}
Let $L$ be a BL-algebra and let $A$ be the MV-center of $L$, then; \\
There exists an ordered-preserving group isomorphism (o-isomorphism) from $G_L$ onto $G_A\times_{lex} \mathfrak{S}(L)$ (the lexicographic product), where $\mathfrak{S}(L)$ is the lattice ordered group above.\\
\end{cor}
\begin{proof}
Consider $\Theta: G_L\to G_A\times \mathfrak{S}(L)$ defined by: $$\Theta([\mathbf{a},\mathbf{b}]=([\overline{\overline{\mathbf{a}}},\overline{\overline{\mathbf{b}}}],[\mathbf{a}+\overline{\overline{\mathbf{b}}},\mathbf{b}+\overline{\overline{\mathbf{a}}}])$$
It follows from the equations in the proof of Theorem \ref{summand} that $\Theta$ is a well-defined group isomorphism. In addition, suppose $[\mathbf{a}, \mathbf{b}]\preceq [\mathbf{c}, \mathbf{d}]$, then there exists $\mathbf{k}\in M_L$ such that $\mathbf{a}+\mathbf{d}+\mathbf{k}\leq \mathbf{b}+\mathbf{c}+\mathbf{k}$. 
Hence, $\overline{\overline{\mathbf{a}}}+\overline{\overline{\mathbf{d}}}+\overline{\overline{\mathbf{k}}}\leq \overline{\overline{\mathbf{b}}}+\overline{\overline{\mathbf{c}}}+\overline{\overline{\mathbf{k}}}$, and so $[\overline{\overline{\mathbf{a}}},\overline{\overline{\mathbf{b}}}]\preceq [\overline{\overline{\mathbf{c}}},\overline{\overline{\mathbf{d}}}]$. In addition, if $[\overline{\overline{\mathbf{a}}},\overline{\overline{\mathbf{b}}}]=[\overline{\overline{\mathbf{c}}},\overline{\overline{\mathbf{d}}}]$, then there exists $\mathbf{u}\in M_L$ such that $\overline{\overline{\mathbf{a}}}+\overline{\overline{\mathbf{d}}}+\mathbf{u}=\overline{\overline{\mathbf{b}}}+\overline{\overline{\mathbf{c}}}+\mathbf{u}$. Hence, $\mathbf{a}+\overline{\overline{\mathbf{b}}}+\mathbf{d}+\overline{\overline{\mathbf{c}}}+\mathbf{k}+\mathbf{u}\leq \mathbf{b}+\overline{\overline{\mathbf{a}}}+\mathbf{c}+\overline{\overline{\mathbf{d}}}+\mathbf{k}+\mathbf{u}$, and $[\mathbf{a}+\overline{\overline{\mathbf{b}}},\mathbf{b}+\overline{\overline{\mathbf{a}}}]\preceq [\mathbf{c}+\overline{\overline{\mathbf{d}}},\mathbf{d}+\overline{\overline{\mathbf{c}}}]$. Therefore, $\Theta([\mathbf{a},\mathbf{b}]\leq_l\Theta([\mathbf{c},\mathbf{d}]$, where $\leq_{lex}$ is the lexicographic order on $G_A\times \mathfrak{S}(L)$. 
\end{proof}
\begin{rem}
The isomorphism $\Theta$ of Corollary \ref{isom} is not in general a lattice-isomorphism. A general justification of this follows from the fact that the lexicographic product of two $\ell$-groups is not an $\ell$-group, unless the first component is an $o$-group \cite[Prop. 7]{KM}.
\end{rem}
For BL-algebras of cancellative type, several aspects of the construction can be simplified as we will see next.
\begin{prop}\label{cancel}
If $L$ is a BL-algebra of cancellative type, then $(M_L,+)$ is a cancellative monoid, that is for every good sequences $\mathbf{a}, \mathbf{b}, \mathbf{c}$ in $L$ such that $\mathbf{a}+\mathbf{c}=\mathbf{b}+\mathbf{c}$, then $\mathbf{a}=\mathbf{b}$.
\end{prop}
\begin{proof}
Since $L$ is a subdirect product of BL-chains of cancellative type, it is enough to prove the result for BL-chains of cancellative type. Let $\mathbf{a}, \mathbf{b}, \mathbf{c}$ be good sequences of a BL-chain $C$ such that $\mathbf{a}+\mathbf{c}=\mathbf{b}+\mathbf{c}$. Then by Lemma \ref{goodschain}, there exist integers $p,q,r\geq 0$ and $a,b,c\in C$ such that $\mathbf{a}=(1^p,a)$, $\mathbf{b}=(1^q,b)$, $\mathbf{c}=(1^r,c)$. So, $(1^{p+r},a+c, a\otimes c)=(1^{q+r}, b+c, b\otimes c)$. If $p=q$, then since $C$ is of cancellative type, then $a=b$ and $\mathbf{a}=\mathbf{b}$. If $p\leq q-2$, then from $(1^{p+r},a+c, a\otimes c)=(1^{q+r}, b+c, b\otimes c)$, we obtain $a\otimes c=1$. Hence, $a=c=1$, which is a contradiction. If $p=q-1$, then $c\leq b+c=a\otimes c$and $a+c=1$. Hence, $a\otimes c=c\otimes 1$ and $a+c=c+1$ and since $C$ is of cancellative type, we obtain $a=1$. Hence $b+c=c$, which implies by Lemma \ref{goodchain}(iv) that $b=0$ or $c=1$. In either case, there is a contradiction. By symmetry the case $p>q$ leads to similar contradictions.
\end{proof}
Note that for BL-algebras of cancellative type, the equivalence relation $\sim$ on $M_L\times M_L$ can be simplified as follows. $(\mathbf{a}, \mathbf{b})\sim (\mathbf{c}, \mathbf{d})$ if and only if $\mathbf{a}+\mathbf{d}=\mathbf{b}+\mathbf{c}$. That is, $[\mathbf{a}, \mathbf{b}]=[\mathbf{c}, \mathbf{d}]$ if and only if $\mathbf{a}+\mathbf{d}=\mathbf{b}+\mathbf{c}$. In addition, for such BL-algebras, the order relation $\preceq$ is simplified as $[\mathbf{a}, \mathbf{b}]\preceq [\mathbf{c}, \mathbf{d}]$ if and only if $\mathbf{a}+\mathbf{d}\leq \mathbf{b}+\mathbf{c}$. \\
\begin{prop}
For every BL-algebra $L$ of cancellative type, $(M_L,+, \leq)$ is submonoid of the positive cone of $G_L$ and the restriction of the order of $G_L$ to $M_L$ coincides with $\leq$.
\end{prop}

\section{Examples}
This section is devoted to computations of the Chang's $\ell$-group of some of the most important BL-algebras. Besides MV-chains, linearly ordered product algebras and linearly ordered G\"odel algebras are the most important classes of BL-algebras. We compute their Chang's $\ell$-group next.\\
The following facts simplify the computation of the Chang's $\ell$-group for BL-chains.
\begin{rem} \label{grp-chain}
Suppose that $L$ is a BL-chain and let $[\mathbf{a}, \mathbf{b}]\in G_L$. Then by Proposition \ref{goodschain} there exists $p,q\geq 0$, $a,b\in L$ such that $\mathbf{a}=(1^p,a)$ and $\mathbf{b}=(1^q,b)$. If $\mathbf{a}=(0)$, then $[\mathbf{a}, \mathbf{b}]=[(0), \mathbf{b}]=[(1), (1^{q+1},b)]$. Similarly if $\mathbf{b}=(0)$, then $[\mathbf{a}, \mathbf{b}]=[(1^{p+1}, a), (1)]$. Therefore every element $[\mathbf{a}, \mathbf{b}]\in G_L$ has the form $[(1^p,a), (1^q,b)]$ with $a\ne 0$ and $b\ne 0$.\\
In addition, for a BL-chain $L$, the $\ell$-group $\mathfrak{S}(L)$ can be described as follows. $\mathfrak{S}(L)=\{[(a),(b)]: a,b\in L\; \text{and}\; \overline{\overline{a}}=\overline{\overline{b}}\}$.
\end{rem}
We can apply the above to compute the Chang's $\ell$-group of linearly ordered G\"odel algebras
\begin{prop}\label{godel}
The Chang's $\ell$-group of any linearly ordered G\"odel algebra is isomorphic to the $o$-group $\mathbb{Z}$.
\end{prop}
\begin{proof}
First, note that if $L$ is a linearly ordered G\"odel algebra, then for $y<x$, $x\to y=x$ \cite[Lemma 4.2.14]{H2}. It follows that $y<x$ implies $y=x\wedge y=x\otimes (x\to y)=x\otimes y$. Thus, $x\otimes y=Min(x,y)$. Also in $L$, $x<y$ implies $x\to y=y$ \cite[Lemma 4.2.14]{H2}, in particular the MV-center of $L$ is the two-element Boolean algebra $\mathbf{2}:=\{0,1\}$. It is known and easy to see that $G_{\mathbf{2}}\cong \mathbb{Z}$. One should also note that for every $a,b\in L\setminus \{0\}$, $a+b=1$.
In addition, from the description of $\mathfrak{S}(L)$ given in Remark \ref{grp-chain}, we have $\mathfrak{S}(L)=\{[(a),(b)]: a,b>0\}$. But, for every $a,b>0$, $(a)+(\alpha)=(b)+(\alpha)$, with $\alpha=Min(a,b)$. Therefore, $[(a), (b)]=[(0),(0)]$, $\mathfrak{S}(L)=\{0\}$, and by Theorem \ref{summand} $G_L=G_A\cong \mathbb{Z}$.
\end{proof}
\begin{lem}\label{chain-group}
1. If $L$ is linearly ordered BL-algebra, then $G_L$ is an $o$-group.\\
2. If $L$ is a BL-algebra of cancellative type so that $G_L$ is an $o$-group, then $L$ is a BL-chain.
\end{lem}
\begin{proof}
Suppose that $L$ is a BL-chain, then elements of $M_L$ has the form $(1^p,a)$ with $p\geq 0$ integer and $a\in L$. It follows that $M_L$ is a chain, and therefore by construction of $G_L$ and $\preceq$, we conclude that $G_L$ is linearly ordered. Conversely, if $G_L$ is linearly ordered, and $a,b\in L$, then $[(a),(0)]\preceq [(b),(0)]$ or $[(b),(0)]\preceq [(a),(0)]$. Hence, $(a)\leq (b)$ or $(b)\leq (a)$, that is $a\leq b$ or $b\leq a$. Therefore, $L$ is linearly ordered as claimed.
\end{proof}
Now, we can compute the Chang's $\ell$-group of the linearly ordered product algebras, starting with $[0,1]$.
\begin{prop}\label{group-p}
Let $L=[0,1]$ with the product structure. Then $G_L$ is isomorphic to $\mathbb{Z}\times_{lex} \mathbb{R}_{+}$.
\end{prop}
\begin{proof}
First, note that if $L$ is a BL-chain such that $G_L$ is an $o$-group, then the isomorphism $\Theta$ of Corollary \ref{isom} is a lattice isomorphism, and $G_L\cong G_A\times_{lex} \mathfrak{S}(L)$ (as $o$-groups).
Since $G_L$ is linearly ordered by Lemma \ref{chain-group}, and $MV(L)=\mathbf{2}$, it is enough to show that $\mathfrak{S}(L)\cong \mathbb{R}_{+}$. One should note that for every $a,b\in L\setminus \{0\}$, $a+b=1$. 
In addition, as above, $\mathfrak{S}(L)=\{[(a), (b)]: a,b>0\}$. Consider $\phi: \mathfrak{S}(L)\to \mathbb{R}_{+}$ defined by $\phi ([(a), (b)])=\dfrac{a}{b}$. 
 It is clear that $[(a), (b)]=[(x), (y)]$ if and only if $\dfrac{a}{b}=\dfrac{x}{y}$. So, $\phi$ is a well-defined injective map. For the surjection of $\phi$, let $x\in \mathbb{R}_{+}$. Then, there exists an integer $n>1$ such that $0\leq \dfrac{x}{n}<1$, and it follows that $\phi ([(\dfrac{x}{n}),(\dfrac{1}{n})])$. Finally, it is a routine verification to check that $\phi$ is an order-preserving group homomorphism.
\end{proof}
For product BL-chains, the group $\mathfrak{S}(L)$ coincides with the well-known $o$-group of a product BL-chain as treated in \cite{C3, H2} .
  \begin{thm}\cite[Thm. 4.1.8]{H2}\label{Ha}
 For every product BL-chain $(L, \wedge, \vee, \otimes, \rightarrow, 0, 1)$, there exists a unique (up to isomorphism) linearly ordered Abelian group $\mathfrak{H}(L)=(G, +_G, 0_G, \leq_G)$ such that for every $g,h\in L\setminus \{0\}$:\\
 (i) $0_G=1$;\\
 (ii) $L\setminus \{0\}$ is equal to the negative cone of $\mathfrak{H}(L)$;\\
 (iii) $g+_Gh=g\otimes h$;\\
 (iv) $g\leq_Gh$ if and only if $g\leq h$.
 \end{thm}
 For the convenience of the reader, we outline the construction of $\mathfrak{H}(L)$. Since $L$ is a product BL-chain, then $(L\setminus\{0\}, \otimes, 1)$ is a cancellative monoid \cite[Lemma 5.1]{C3}. We start by the relation $\equiv$ on $L\setminus\{0\}\times L\setminus\{0\}$ defined by $(x,y)\equiv (a,b)$ if and only if $x\otimes b=a\otimes y$, which is clearly an equivalence relation. Then let $G=L\setminus\{0\}\times L\setminus\{0\}/\equiv$. If one denotes the class of $(a,b)$ by $[a,b]$ and equips $G$ with the operation:$$[a,b]+_G[c,d]=[a\otimes c, b\otimes d], \\
 -[a,b]=[b,a],\\
 0_G=[1,1]$$
 Then $\mathfrak{H}(L)=(G, +_G, 0_G, \leq_G)$ is a totally ordered Abelian group, where $[a,b]\leq_G[c,d]$ if and only if $a\otimes d\leq b\otimes c$. Note that $(L\setminus\{0\}, \otimes, 1)$ is isomorphic to the submonoid of $\mathfrak{S}(L)$ made of elements of the form $[a,1]$. Clearly under this identification the conditions (i)-(iv) of Theorem \ref{Ha} are obvious.\\
 In addition, the groups $\mathfrak{H}(L)$ and $\mathfrak{S}(L)$ are naturally isomorphic. In fact $[a,b]\mapsto [(a),(b)]$ is an order-preserving isomorphism from $\mathfrak{H}(L)$ onto $\mathfrak{S}(L)$.\\
 It is easy to see that if $L=[0,1]$ is equipped with the product structure, then $\mathfrak{H}(L)\cong (\mathbb{R}_{+}, \cdot, 1)$. Therefore, the next result is a generalization of Proposition \ref{group-p}.
 \begin{prop}
 If $L$ is a product BL-chain, then the Chang's group of $L$ is naturally isomorphic to $\mathbb{Z}\times_{lex} \mathfrak{S}(L)$.
 \end{prop}
 \begin{proof}
 Recall that every element $[\mathbf{a}, \mathbf{b}]\in G_L$ has the form $[(1^p,a), (1^q,b)]$ with $a\ne 0$ and $b\ne 0$. Now, consider the map $\varphi: \mathfrak{S}(L)\to \mathbb{Z}\times \mathfrak{S}(L)$ defined by: $\varphi ([(1^p,a), (1^q,b)])=(p-q, a-b)$ where $a-b:=a+_G(-b)$. We have, $[(1^p,a), (1^q,b)]=[(1^r,x), (1^s,y)]$ if and only if $(1^{p+s+1}, a\otimes y)=(1^{q+r+1}, b\otimes x)$ if and only if $p+s=q+r$ and $a\otimes y=b\otimes x$ if and only if $p-q=r-s$ and $a+_Gb=x+_Gy$ if and only if $p-q=r-s$ and $a-b=x-y$. Therefore, $\varphi$ is a well-defined one-to-one map that is clearly a homomorphism. For the surjectivity of $\varphi$, let $(m,x)\in \mathbb{Z}\times \mathfrak{S}(L)$. Then $(m,x)=\varphi([(1^m,x),(0)])$ if $m\geq 0$ and $0\leq_Gx$; $(m,x)=\varphi([(1^m),(-x)])$ if $m\geq 0$ and $x\leq_G0$; $(m,x)=\varphi([(x),(1^{-m})])$ if $m\leq 0$ and $0\leq_Gx$; $(m,x)=\varphi([(0),(1^{-m},-x)])$ if $m\leq 0$ and $x\leq_G0$. Finally, if it is clear that $[(1^p,a), (1^q,b)]\preceq [(1^r,x), (1^s,y)]$ if and only if $p-q< r-s$ or $p-q=r-s$ and $a-b\leq_Gx-y$. Therefore, $\varphi$ is an isomorphism of lattice ordered Abelian groups.
 \end{proof}
\section{A nearly adjoint pair of functors}
In this section, we consider the natural extension of the functor $\Xi$ (the inverse of the Mundici's functor) treated in \cite{CM} to BL-algebras.\\
Let $\mathcal{BL}$ denotes the category of BL-algebras, and $\mathcal{LU}$ denotes the category of lattice ordered groups with strong units. Recall that if $\langle G, u\rangle $ and $\langle G',u'\rangle$ are $\ell$-groups with strong units, a morphism $\phi: \langle G, u\rangle \to \langle G',u'\rangle$ is a group homomorphism $\phi:G\to G'$ such that $\phi(u)=u'$. \\
1. Now consider $\Xi:\mathcal{BL}\to \mathcal{LU}$ defined:
\begin{itemize}
\item On objects: For every BL-algebra $L$, $\Xi(L)$ is the lattice ordered group with strong unit $\langle G_L,u_L\rangle$ of Proposition \ref{strong}.
\item On morphisms:  If $X,Y$ are BL-algebras, and $f:X\to Y$ is a homomorphism of BL-algebras. Define $\Xi(f):\langle G_X, u_X\rangle \to \langle G_Y, u_Y\rangle$ by: $\Xi(f)([\mathbf{a},\mathbf{b}])=[f(\mathbf{a}),f(\mathbf{b})]$, where $f(\mathbf{a}):=(f(a_1),f(a_2),\ldots)$.
\end{itemize}
Then, one can verify that $\Xi(f)$ is a well-defined group homomorphism and $\Xi(f)(u_X)=u_Y$, that is $\Xi(f)$ is a morphism in $\mathcal{LU}$. It is a routine verification that $\Xi$ is a functor from $\mathcal{BL}$ to $\mathcal{LU}$. \\
2. On the other hand, we view the Mundici's functor $\Gamma$ as having for codomain $\mathcal{BL}$. More precisely, consider $\Gamma:\mathcal{LU}\to \mathcal{BL}$ defined:  
\begin{itemize}
\item On objects: For every $\ell$-group with strong unit $\langle G,u\rangle$, $\Gamma(\langle G,u\rangle)=([0,u], \vee, \wedge, \otimes, \to, 0, u)$; with $x\otimes y=u-(2u-x-y)\wedge u$ and $x\to y=(u-x+y)\wedge u$.
\item On morphisms: If $\langle G, u\rangle $ and $\langle G',u'\rangle$ are $\ell$-groups with strong units and $\phi: \langle G, u\rangle \to \langle G',u'\rangle$ a morphism. Define $\Gamma(\phi)$ to be the restriction of $\phi$ to $[0,u]$.
\end{itemize}
It can be proved that $([0,u], \vee, \wedge, \otimes, \to, 0, u)$ is a BL-algebra (in fact, an MV-algebra)\cite[Prop. 1.1] {CM}and $\Gamma$ is a functor from $\Gamma:\mathcal{LU}\to \mathcal{BL}$.\\
We will need the following key results.
\begin{lem}\cite[Lem. 1.5]{CM}
Let $\langle G,u\rangle$ be an $\ell$-group with strong unit, and $A=\Gamma(G,u)\subseteq G$. For every $a\in G^{+}$, there exists a unique good sequence $g(a):=(a_1,\ldots,a_n)$ in $A$ such that $a=a_1+\cdots+a_n$.
\end{lem}\cite[Cor. 3.5]{CM}
\begin{prop}
For every $\ell$-group with strong unit $\langle G,u\rangle$, define the map $\psi_{\langle G,u\rangle}: \langle G,u\rangle\to \Xi\Gamma(G,u)$ by $\psi_{\langle G,u\rangle}(a)=[g(a^+),g(a^-)]$, using the notations of the preceding Lemma.\\
Then $\psi_{\langle G,u\rangle}$ is an isomorphism in $\mathcal{LU}$.
\end{prop}
Next, we describe the best relationship between $\Gamma$ and $\Xi$. For basic terminologies, notations, and facts about functor adjunctions, one should consult for example \cite[Ch. 4]{Sa}\\
First, the functor $\Xi\Gamma$ is naturally equivalent to $1_{\mathcal{LU}}$ \cite[Thm. 3.6]{CM}.\\
In addition, we show that there is a natural transformation $\eta: 1_{\mathcal{BL}}\to \Gamma\Xi$ such that for every morphism $f:L\to \Gamma (G,u)$, there exists a morphism $g:\Xi L\to \langle G,u\rangle $ such that $\Gamma(g)\eta_L=f$.\\
Consider $\eta: 1_{\mathcal{BL}}\to \Gamma\Xi$ described as follows: for every BL-algebra $L$, $\eta_L:L\to \Gamma \Xi L$ is defined by $\eta_L(a)=[(\bar{\bar{a}}),(0)]$. Then $\eta_L$ is well-defined since for every $a\in L$, it is clear that $[(\bar{\bar{a}}),(0)]\in [0,u_L]$. To see that each $\eta_L$ is a homomorphism in $\mathcal{BL}$, we can observe that every morphism $h:MV(L)\to M$ in $\mathcal{MV}$ extends uniquely to a morphism $\bar{f}:L\to M$ in $\mathcal{BL}$. More precisely, $\bar{f}$ is defined by $\bar{f}(a)=f(\bar{\bar{a}})$. Now recall \cite[Thm. 2.12]{CM} that for very MV-algebra $A$, the map $\varphi_A:A\to [0,u_A]$ defined by $\varphi(a)=[(a),(0)]$ is an isomorphism. Then $\eta_L=\iota_A\circ \overline{\varphi_A}$ where $A=MV(L)$ and $\iota_A$ is the inclusion $[0,u_A]\hookrightarrow G_A\hookrightarrow G_L$.
It is straightforward to verify that $\eta$ is a natural transformation.\\
On the other hand, let $f:L\to \Gamma (G,u)$ be a morphism in $\mathcal{BL}$, then $\psi_{\langle G,u\rangle}^{-1}\Xi(f)$ is a morphism from $\Xi L\to \langle G, u\rangle $ in $\mathcal{LU}$. Setting $g=\psi_{\langle G,u\rangle}^{-1}\Xi(f)$, it is readily verified that $\Gamma(g)\eta_L=f$. The only thing missing to obtain an adjunction is the uniqueness of $g$. Unfortunately, $g$ is not unique in general. One might think there could have been a more judicious choice for $\eta$, or on the order of functors. But as we can see in the example below, there is no possible way to get an adjoint pair from these two functors.\\
We will use the following lemma.
\begin{lem} \label{hom}
Let $\mathbf{G}=\langle G,u\rangle $ and $A=\Gamma \mathbf{G}$ be an MV-algebra. Let $L$ be a BL-algebra whose MV-center is the two element Boolean algebra.\\
1. Let $f:A\to L$ be a morphism in $\mathcal{BL}$. Then for every $a\in A$, if $2a\leq u$, then $f(a)=0$ and if $u\leq 2a$, then $f(a)=1$.\\
2. There is a unique morphism $g:L\to A$.
\end{lem}
\begin{proof}
1. Since $A$ is an MV-algebra, so is $f(A)$. Thus $f(A)=\{0,1\}$. It follows that, for every $a\in A$, $f(a)\ne f(u-a)$ as $f(u-a)=\overline{f(a)}$. Now, if $2a\leq u$, then $a\leq u-a$, so $f(a)<f(u-a)$, from which we obtain $f(a)=0$. The case $1\leq 2a$ is argued similarly. \\
2. More generally, for any MV-algebra $A$ and any BL-algebra $L$, $Hom(L,A)\simeq Hom(MV(L),A)$. In fact every morphism $g:MV(L)\to A$ extends uniquely to a morphism $h:L\to A$, where $h(x)=g(\overline{\overline{x}})$.

\end{proof}
\begin{ex}
Let $L=[0,1]$ with the product structure and $\mathbf{G}=\langle \mathbb{Z}\times_{lex}\mathbb{R}_{+},(1,1)\rangle$. We know from Proposition \ref{group-p} that $\Xi L=\mathbf{G}$. We claim that there is no bijection between $Hom(\Xi L, \mathbf{G})$ and $Hom(L, \Gamma\mathbf{G})$ and there is no bijection bewteen $Hom(\Gamma \mathbf{G}, L)$ and $Hom(\mathbf{G}, \Xi L)$. In fact, we have the following information about each of the four sets.\\
(i) $Hom(L, \Gamma\mathbf{G})$ is a singleton by Lemma \ref{hom}(2).\\
(ii) $Hom(\Xi L, \mathbf{G})$ has at least two elements: the identity map, and $f:\Xi L\to \mathbf{G}$ defined by $f(m,x)=(m,1)$.\\
(iii) $Hom(\Gamma \mathbf{G}, L)$ is a singleton. First, note that $\Gamma \mathbf{G}=[(0,1);(1,1)]=\{0\}\times [1,\infty) \biguplus \{1\}\times (0,1]$. On the other hand, suppose that $f:\Gamma \mathbf{G}\to L$ is a morphism. Then by Lemma \ref{hom}(1), $f(\{0\}\times [1,\infty))=\{0\}$ and $f(\{1\}\times (0,1])=\{1\}$. In fact, one can easily verify that the map $f:\Gamma \mathbf{G}\to L$ defined by:
$$f(m,x) =
\left\{\begin{array}{ll}
  0 & ,\ \ \mbox{if} \ \ (m,x)\in \{0\}\times [1,\infty)\\
 1 &  ,\ \ \mbox{if}\ \ (m,x)\in \{1\}\times (0,1]
\end{array}\right.$$
is a morphism.\\
(iv) $Hom(\mathbf{G}, \Xi L)=Hom(\Xi L, \mathbf{G})$, which by case (ii) has at least two elements.
\end{ex}
\section{Conclusion and final Remarks}
The main goal of this work was to construct a group associated to any BL-algebra, that generalizes the well-known Chang's $\ell$-group of an MV-algebra. This was achieved using mostly the same techniques used by the authors of \cite{CM}, but several proofs had to be reinvented completely in the BL-algebras setting. One can summarize most of the work by the functor $\Xi$ from the category $\mathcal{BL}$ of BL-algebras to the category $\mathcal{LU}$ of Abelian $\ell$-groups with strong units. There are at least two aspects of the study of BL-algebras that could be enhanced with the present work. On one hand, the failure of $\Xi$ to be a functor offers a new measure of the understanding of the gap between MV-algebras and BL-algebras. On the other hand, the functor $\Xi$ can be used to investigate classification problems of BL-algebras, just like the homology (and homotopy) are used to classify topological spaces. For example, we saw in Proposition \ref{godel} that the Chang's $\ell$-group of any linearly ordered G\"odel algebra is isomorphic to the $o$-group $\mathbb{Z}$. One could ask if this characterizes completely the G\"odel BL-chains. We plan to use the fucntor $\Xi$ to investigate isomorphism problems of that type and more.

\end{document}